\documentclass[10pt]{amsart}
\usepackage{amsfonts}
\usepackage{ifthen}
\usepackage{amsthm}
\usepackage{amsmath}
\usepackage{graphicx}
\usepackage{amscd,amssymb,amsthm}

\newcounter{minutes}
\divide\time by 60
\newcounter{hours}
\multiply\time by 60 \addtocounter{minutes}{-\time}
\newtheorem{theorem}{Theorem}
\newtheorem{corollary}{Corollary}
\newtheorem{lemma}{Lemma}

\keywords{Analytic functions, Univalent functions, Integral Operator, Generalized Bessel functions, Ahlfors-Becker
univalence criteria, fractional differential operator} 
\subjclass[2010]{Primary 33C10, 30C45;Secondary 30C20, 30C75.}
\email{halkarasani@ud.edu.sa}
\email{abeer.zahrani@windowslive.com}
\email{ssalhajri@ud.edu.sa}
\input{tcilatex}

\begin{document}
\def\thefootnote{}
\def\thefootnote{\@arabic\c@footnote}

\title{Univalence of Criteria for Linear Fractional Differential Operator $%
D_{\lambda }^{n,\alpha }$\ with the Bessel Functions}
\author{H.A. Al-Kharsani}
\address{Department of Mathematics,College of Sciences, University Of
Dammam, 838 Dammam, Saudi Arabia}
\author{Abeer M. Al-Zahrani}
\address{}
\author{S.S. Al-Hajri}
\address{}
\maketitle

\begin{abstract}
In this paper our aim is to extend and improve the sufficient conditions for
integral operators involving the normalized forms of the generalized Bessel
functions of the first kind to be univalent in the open unit disk as
investigated recently by \textit{(Erhan, E. Orhan, H. and Srivastava, H.
(2011). Some sufficient conditions for univalence of certain families of
integral operators involving generalized Bessel functions. Taiwanese Journal
of Mathematics, 15 (2), pp.}883-917\textit{)} and \textit{(Baricz, \textsc{%
\'{A}}. and Frasin, B. (2010). Univalence of integral operators involving
Bessel functions. Applied Mathematics Letters, 23 (4), pp.371--376).}
\end{abstract}

\footnotetext{\texttt{File:~\jobname .tex, printed: \number\year-0\number%
\month-\number\day, \thehours.\ifnum\theminutes<10{0}\fi\theminutes}} 
\makeatletter

\makeatother

\section{\textbf{Introduction and some preliminary results}}

\setcounter{equation}{0}

Several applications of Bessel functions arise naturally in a wide variety
of problems in applied mathematics, statistics, operational research,
theoretical physics and engineering sciences . Bessel functions are series
solutions to a second order differential equation that ascend in many and
diverse situations. Bessel's differential equation of order $\nu $ is
defined as (see, for details, \cite{22}):

\begin{equation}
z^{2}w^{2}+bzw+[dz^{2}-~\nu ^{2}+(1-b))~\nu ]w=0\;\;b,d,\nu \in \mathbb{C}
\label{eq1.1}
\end{equation}

A particular solution of the differential equation $\left( \ref{eq1.1}%
\right) $, which is denoted by $w_{v~,b,d}(z)\ $is called the generalized
Bessel function of the first kind of order $\nu $. In fact, we have the
following familiar series representation for the function $w_{v~,b,d}(z)~$:%
\begin{equation}
\displaystyle~w_{v~,b,d}(z)~=\sum_{n=0}^{\infty }\frac{(-d)^{n}}{n!\Gamma
(~\nu +n+\frac{b+1}{2}1)}.\left( \frac{z}{2}\right) ^{2n+~\nu }(z\in \mathbb{%
C).}  \label{eq2.1}
\end{equation}

\bigskip where $\Gamma (z)$ stands for the Euler gamma function. The series
in $\left( \ref{eq2.1}\right) $ permits us to study the Bessel, the modified
Bessel and the spherical Bessel functions in a unified manner. Each of these
particular cases of the function $w_{v~,b,d}(z)~$ is worthy of mention here.

\textbullet\ For $b=d=1$ in $\left( \ref{eq2.1}\right) $, we obtain the
familiar Bessel function $~J_{v}(z)$ defined by (see \cite{22}; see also 
\cite{6a})

\begin{equation}
\displaystyle~J_{v}(z)~=\sum_{n=0}^{\infty }\frac{(-1)^{n}}{n!\Gamma (v+n+1)}%
.\left( \frac{z}{2}\right) ^{2n+v}(z\in \mathbb{C).}  \label{eq3.1}
\end{equation}

\textbullet\ For$b=-d=1$ in $\left( \ref{eq2.1}\right) $, we obtain the
familiar Bessel function$~I_{v}(z)$defined by (see \cite{22}; see also \cite%
{6a})

\begin{equation}
\displaystyle~I_{v}(z)~=\sum_{n=0}^{\infty }\frac{1}{n!\Gamma (v+n+1)}%
.\left( \frac{z}{2}\right) ^{2n+v}(z\in \mathbb{C).}  \label{eq4.1}
\end{equation}

Now, consider the function $u_{v,b,d}:\mathbb{C}\rightarrow \mathbb{C},$
defined by the transformation 
\begin{equation*}
u_{v,b,d}(z)=2^{v}\Gamma \left( v+\frac{b+1}{2}\right) \ z^{-v/2}w_{v,b,d}(%
\sqrt{z}).
\end{equation*}%
By using the well-known Pochhammer (or Appell ) symbol, defined in terms of
the Euler gamma function, 
\begin{equation*}
(a)_{n}=\frac{\Gamma (a+n)}{\Gamma (a)}=a(a+1)\ldots (a+n-1)
\end{equation*}%
and $(a)_{0}=1,$ we obtain for the function $u_{v,b,d}$ the following
representation 
\begin{equation*}
\displaystyle u_{v,b,d}(z)=\sum_{n\geq 0}\frac{(-d/4)^{n}}{{\left( v+\frac{%
b+1}{2}\right) }_{n}}\frac{z^{n}}{n!},
\end{equation*}%
where $v+\frac{b+1}{2}\neq 0,-1,-2,{\dots }.$ This function is analytic on $%
\mathbb{C}$ and satisfies the second order linear differential equation 
\begin{equation*}
4z^{2}u^{\prime \prime }(z)+2(2v+b+1)zu^{\prime }(z)+d\ zu(z)=0.
\end{equation*}%
We now introduce the function $\varphi _{v,b,d}(z)$ defined in terms of the
generalized Bessel function $w_{v,b,d}(z)$ by%
\begin{eqnarray*}
\varphi _{v,b,d}(z) &=&zu_{v,b,d}(z) \\
&=&2^{v}\Gamma \left( v+\frac{b+1}{2}\right) \ z^{1-v/2}w_{v,b,d}(\sqrt{z})
\\
&=&z+\sum_{n=1}^{\infty }\frac{(-d)^{n}z^{n+1}}{4^{n}n!(k)_{n}}.\ \ \ \
\left( k=v+\frac{b+1}{2}\right)
\end{eqnarray*}

Let $A$ denote the class of analytic function $f$ defined in the unit disk $%
\displaystyle U=\{z:|z|<1\}$and has the form $\displaystyle %
f(z)=z+\sum_{k=2}^{\infty }a_{k}z^{k}.$For functions $\displaystyle %
f(z)=z+\sum_{k=2}^{\infty }a_{k}z^{k}$ and $g(z)=z+\sum_{k=2}^{\infty
}b_{k}z^{k},$ the Hadamard product (or convolution) $f\ast g$ is defined, as
usual, by$\displaystyle(f\ast g)(z)=z+\sum_{k=2}^{\infty }a_{k}b_{k}z^{k}.$

\bigskip This paper deals with the linear fractional differential operator $%
D_{\lambda }^{n,\alpha }$ for complex numbers $\alpha ,\lambda $ ,where 
\begin{equation*}
D_{\lambda }^{n,\alpha }f(z)=\underbrace{[(D_{\lambda }^{1,\alpha }\ast
D_{\lambda }^{1,\alpha }\ldots \ast D_{\lambda }^{1,\alpha })]}%
_{n-times}\ast f(z))
\end{equation*}%
\begin{equation*}
D_{\lambda }^{1,\alpha }=\Gamma (2-\alpha )z^{\alpha }D_{z}^{\alpha
}f(z)\ast g_{\lambda }(z),\,\,\alpha \neq 2,3,4,...
\end{equation*}%
\begin{eqnarray*}
g_{\lambda }(z) &=&\frac{z-(1-\lambda )z^{2}}{(1-z)^{2}}=z+\sum_{k=2}^{%
\infty }[1+\lambda (k-1)]z^{k},\  \\
D_{z}^{\alpha }f(z) &=&\frac{1}{\Gamma (1-\alpha )}\frac{d}{dz}%
\dint\limits_{0}^{z}\frac{f(t)}{(z-t)^{\alpha }}dt,0\leq \alpha <1
\end{eqnarray*}%
The operator $D_{\lambda }^{n,\alpha }$ was interoduced by \cite{3}. Using
the fractional derivative of order $\alpha ,D_{z}^{\alpha }\ $\cite{11b},
Owa and Srivastava \cite{12} introduced the operator $\Omega ^{\alpha
}:A\rightarrow A,$ which is known as an extension of fractional derivative
and fractional integral, as follows\newline
\begin{eqnarray*}
\Omega ^{\alpha }f(z) &=&\Gamma (2-\alpha )z^{\alpha }D_{z}^{\alpha
}f(z),\,\,\alpha \neq 2,3,4 \\
&=&z+\sum_{k=2}^{\infty }\frac{\Gamma (k+1)\Gamma (2-\alpha )}{\Gamma
(k+1-\alpha )}a_{k}z^{k} \\
&=&\varphi (2,2-\alpha ;z)\ast f(z)
\end{eqnarray*}%
In \cite{3}, the linear fractional differential operator $D_{\lambda
}^{n,\alpha }:A\rightarrow A$ is defined as follows,\newline
\begin{equation}
D_{\lambda }^{n,\alpha }f(z)=z+\sum_{k=2}^{\infty }\left( \frac{\Gamma
(k+1)\Gamma (2-\alpha )}{\Gamma (k+1-\alpha )}(1+\lambda (k-1))\right)
^{n}a_{k}z^{k}  \label{eq5.1}
\end{equation}%
When $\alpha =0,$ we get Al-Oboudi differential operator \cite{2}, when $%
\alpha =0$ and $\lambda =1,$ we get Salagean differential operator \cite{17}
and when $n=1$ and $\lambda =0,$ we get Owa-Srivastave fractional
differential operator \cite{12}.\newline

We now introduce the linear fractional differential operator $D_{\lambda
}^{n,\gamma }\varphi _{v_{,b,d}}:A\rightarrow A$ 
\begin{equation}
D_{\lambda }^{n,\gamma }\varphi _{v_{,b,d}}(z)=2^{v}\Gamma \left( v+\frac{b+1%
}{2}\right) D_{\lambda }^{n,\gamma }\left[ z^{1-v/2}w_{v,b,d}(\sqrt{z})%
\right]  \label{eq:1}
\end{equation}

Various general families of integral operators were introduced and studied
earlier in Geometric function theory, such as ( see \cite{18},\cite{19},\cite%
{6},\cite{20},\cite{11a},\cite{11} and \cite{15})\newline
\begin{equation}
H_{\alpha _{1},\alpha _{2}}\ldots _{\alpha _{m},\beta }(z)=\left[ \beta
\int_{0}^{z}t^{\beta -1}\Pi _{i=1}^{m}\left( \frac{h_{i}(t)}{t}\right) ^{%
\frac{1}{\alpha _{i}}}dt\right] ^{\frac{1}{\beta }}~~~~~~~~~~~  \label{eq6.1}
\end{equation}%
\begin{equation}
F_{m,\gamma }(z)=\left[ (m\gamma +1)\int_{0}^{z}\Pi _{i=1}^{m}\left(
f_{i}(t)\right) ^{\gamma }dt\right] ^{1/(m\gamma +1)}~~  \label{eq:1.8}
\end{equation}

\begin{equation}
G_{\lambda }(z)=\left[ \lambda \int_{0}^{z}t^{\lambda -1}\left(
e^{g(t)}\right) ^{\lambda }dt\right] ^{\frac{1}{\lambda }}~~  \label{eq7.1}
\end{equation}

where the functions $h_{1},\ldots ,h_{m},f_{1},\ldots ,f_{m}$ and $g$ \
belong to the class $A$ and the parameters $\alpha _{1},\ldots ,\alpha
_{m}\in \mathbb{C}\backslash \left\{ 0\right\} $ and $\beta ,\gamma ,\lambda 
$ are complex numbers such that the integrals in $\left( \ref{eq6.1}\right) $
,$\left( \ref{eq:1.8}\right) $ and $\left( \ref{eq7.1}\right) $ exist. Here
and throughout this paper every many valued function is taken with the
principle branch.\newline

Two of the most important and known univalence criteria for analytic
functions defined in the open unit disk $\mathbb{U}$ were obtained by
Ahlfors \cite{1} and Becker \cite{8} and by Becker (see \cite{9}). Some
extensions of these two univalence criteria were given by Pescar (see \cite%
{14}) involving a parameter $\beta $ and by Pascu (see \cite{13}) involving
two parameters $\alpha $ and $\beta $ . Bulut \cite{10} obtained Sufficient
conditions for the univalence of the integral operator%
\begin{equation}
I_{\beta }^{n,\gamma }(f_{1},\ldots ,f_{m})=\left\{ \beta
\int_{0}^{z}t^{\beta -1}\prod_{i=1}^{m}\left( \frac{D_{\lambda }^{n,\gamma
}f_{i}(t)}{t}\right) ^{\alpha _{i}}dt\right\} ^{\frac{1}{\beta }}
\label{eq8.1}
\end{equation}%
\begin{equation*}
(z\in \mathbb{U}),n\in \mathbb{N}_{0},m\in \mathbb{N},\beta \in \mathbb{C}%
\text{ with }\mathfrak{R}(\beta )>0\ \text{and }\alpha _{i}\in \mathbb{C}%
(i\in \{1,\ldots ,m\}).
\end{equation*}

Recently, szasz and Kupan \cite{21} investigated the univalence of the
normalized Bessel function of the first kind $\phi _{k}:\mathbb{D}%
\rightarrow \mathbb{C},$ defined by \newline
\begin{eqnarray*}
\phi _{k}(z) &=&2^{k}\Gamma (k+1)z^{1-k/2}J_{k}\left( z^{\frac{1}{2}}\right)
\\
\displaystyle~~ &=&z+\sum_{n=1}^{\infty }\frac{(-1)^{n}z^{n+1}}{%
4^{n}n!(k+1)_{n}}.
\end{eqnarray*}

Families of Integral operators of types $\left( \ref{eq6.1}\right) $ and $%
\left( \ref{eq7.1}\right) $ which involve the normalized forms of the
generalized Bessel functions of the first kind have been investigated in
[4-6] and \cite{7} to obtain sufficient conditions for integral operators to
be univalent in the open unit disk. The main object of this paper is to
extend and improve a forementioned results of \cite{11} and \cite{6}. For
this purpose, it is organized as follows. In Section 2, we prove
inequalities of compositions of the linear fractional differential operator $%
D_{\lambda }^{n,\alpha }$ with the Bessel function of the first kind of
order $\upsilon $ in terms of generalized hypergeometric function. In
Section 3 ,we present univalence criteria for compositions of\ the linear
fractional differential operator $D_{\lambda }^{n,\alpha }$\ with the Bessel
function in $\left( \ref{eq6.1}\right) ,\left( \ref{eq7.1}\right) $and $%
\left( \ref{eq8.1}\right) $. In Section 4, we discuss and compare our
special cases with those in \cite{11} and \cite{6}.

\section{Preliminary Lemmas}

\setcounter{equation}{0}

The following result can be obtained by setting $k_{v,c}(z)=D_{\lambda
}^{n,\gamma }\varphi _{v,b,d}(z)\ $in the known result due to (\cite{16},
Theorem 1).

\begin{theorem}
\label{theorem2.3} Let $v\neq -1,-2,-3,...$ and $b,d\in \mathbb{C}$. If $%
D_{\lambda }^{n,\gamma }\varphi _{v,b,d}$ satisfies any one of the following
inequalities:%
\begin{equation}
\left\vert \frac{z^{2}(D_{\lambda }^{n,\gamma }\varphi _{v,b,d})^{^{\prime
}}(z)}{\left[ D_{\lambda }^{n,\gamma }\varphi _{v,b,d}(z)\right] ^{2}}\left( 
\frac{\left( zD_{\lambda }^{n,\gamma }\varphi _{v,b,d}(z)\right) ^{^{\prime
\prime }}}{(D_{\lambda }^{n,\gamma }\varphi _{v,b,d})^{^{\prime }}(z)}-\frac{%
2z(D_{\lambda }^{n,\gamma }\varphi _{v,b,d})^{^{\prime }}(z)}{D_{\lambda
}^{n,\gamma }\varphi _{v,b,d}(z)}\right) \right\vert <1\;\;\;(z\in \mathbb{U}%
),  \label{eq:2.4}
\end{equation}%
\newline
\begin{equation}
\left\vert \frac{\left[ D_{\lambda }^{n,\gamma }\varphi _{v,b,d}(z)\right]
^{2}}{z^{2}(D_{\lambda }^{n,\gamma }\varphi _{v,b,d})^{^{\prime }}(z)}\left( 
\frac{\left[ zD_{\lambda }^{n,\gamma }\varphi _{v,b,d}(z)\right] ^{^{\prime
\prime }}}{(D_{\lambda }^{n,\gamma }\varphi _{v,b,d})^{^{\prime }}(z)}-\frac{%
2z(D_{\lambda }^{n,\gamma }\varphi _{v,b,d})^{^{\prime }}(z)}{D_{\lambda
}^{n,\gamma }\varphi _{v,b,d}(z)}\right) \right\vert <\frac{1}{4}\;\;\;(z\in 
\mathbb{U}),
\end{equation}%
\begin{equation}
\left\vert \frac{\frac{\left[ zD_{\lambda }^{n,\gamma }\varphi _{v,b,d}(z)%
\right] ^{^{\prime \prime }}}{(D_{\lambda }^{n,\gamma }\varphi
_{v,b,d})^{^{\prime }}(z)}-\frac{2z(D_{\lambda }^{n,\gamma }\varphi
_{v,b,d})^{^{\prime }}(z)}{D_{\lambda }^{n,\gamma }\varphi _{v,b,d}(z)}}{%
\frac{z^{2}(D_{\lambda }^{n,\gamma }\varphi _{v,b,d})^{^{\prime }}(z)}{\left[
D_{\lambda }^{n,\gamma }\varphi _{v,b,d}(z)\right] ^{2}}-1}\right\vert <%
\frac{1}{2}\;\;\;(z\in \mathbb{U}),
\end{equation}%
\newline
\begin{equation}
\Re \left\{ \frac{z^{2}(D_{\lambda }^{n,\gamma }\varphi _{v,b,d})^{^{\prime
}}(z)}{\left[ D_{\lambda }^{n,\gamma }\varphi _{v,b,d}(z)\right] ^{2}}\left( 
\frac{\frac{\left[ zD_{\lambda }^{n,\gamma }\varphi _{v,b,d}(z)\right]
^{^{\prime \prime }}}{(D_{\lambda }^{n,\gamma }\varphi _{v,b,d})^{^{\prime
}}(z)}-\frac{2z(D_{\lambda }^{n,\gamma }\varphi _{v,b,d})^{^{\prime }}(z)}{%
D_{\lambda }^{n,\gamma }\varphi _{v,b,d}(z)}}{\frac{z^{2}(D_{\lambda
}^{n,\gamma }\varphi _{v,b,d})^{^{\prime }}(z)}{\left[ D_{\lambda
}^{n,\gamma }\varphi _{v,b,d}(z)\right] ^{2}}-1}\right) \right\}
<1\;\;\;(z\in \mathbb{U}),  \label{eq:2.7}
\end{equation}%
then $D_{\lambda }^{n,\gamma }\varphi _{v,b,d}$ is univalent in $\mathbb{U}$.
\end{theorem}

By using Pochhammer symbol $(\lambda )_{\mu }$ we obtain the following
series representation for $D_{\lambda }^{n,\gamma }$ as follows:\newline
\begin{equation}
\displaystyle D_{\lambda }^{n,\gamma }f(z)=\left\{ 
\begin{array}{l}
z+\displaystyle\sum_{m=2}^{\infty }\left[ \frac{\left( 1+\frac{1}{\lambda }%
\right) _{m-1}(2)_{m-1}}{\left( \frac{1}{\lambda }\right) _{m-1}(2-\gamma
)_{m-1}}\right] ^{n}a_{m}z^{m}\quad \lambda >0 \\ 
\displaystyle z+\sum_{m=2}^{\infty }\left[ \frac{(2)_{m-1}}{(2-\gamma )_{m-1}%
}\right] a_{k}z^{k}\quad \lambda =0%
\end{array}%
\right.  \label{eq2.23}
\end{equation}%
The following notations will be useful in the sequel:%
\begin{gather*}
\delta (k+w)=\delta (k+w,\lambda ,\gamma ,n)=(k+w)(\frac{1}{\lambda }%
+w)^{n}(2-\gamma +w)^{n}\ \ \ ,w\in \left\{ 0,1,2\right\} \\
\beta (d)=\beta (d,\lambda ,m,n)=\left\vert d\right\vert (\frac{1}{\lambda }%
+m)^{n}(m+1)^{n} \\
\mathcal{M=M(}\gamma ,\lambda ,m,n)=(2-\gamma )^{n}(\frac{1}{\lambda }%
+m)^{n}(m+1)^{n} \\
\mathcal{N=N(}\lambda ,n)=2^{n}(\lambda +1)^{n} \\
\alpha ((w\pm 1)k+w)=w\ \mathcal{N}\ \delta (k+1)\pm \mathcal{M}\ k\ \ \
,w\in \left\{ 0,1,2\right\}
\end{gather*}

The following result is mainly based on \cite{7} and is one of the crucial
facts in the proofs of our main results.

\begin{theorem}
\label{theorem2.4}If the parameters $v,b\in \mathbb{R}$ and $d\in \mathbb{C}$
are so constrained that\newline
\begin{equation*}
\displaystyle k=v+\dfrac{b+1}{2}>\max \left\{ 0,\frac{\beta (d)}{4(1+\frac{1%
}{\lambda })^{n}(3-\gamma )^{n}}-1\right\} ,
\end{equation*}%
\newline
then the function\newline
\begin{equation*}
\displaystyle\frac{D_{\lambda }^{n,\gamma }\varphi _{v,b,d}(z)}{z}:\mathbb{U}%
\rightarrow \mathbb{C}
\end{equation*}%
satisfies the following inequality:\newline
\begin{equation}
\displaystyle\frac{4\mathcal{M}\ k\ \delta (k+1)-\alpha (2k+1)\beta (d)+%
\frac{\mathcal{N}}{8}\left[ \beta (d)\right] ^{2}}{\mathcal{M}\ k[4\delta
(k+1)-\beta (d)]}\leqq \left\vert \frac{D_{\lambda }^{n,\gamma }\varphi
_{v,b,d}(z)}{z}\right\vert \leqq \frac{32\left[ \delta (k)\right] ^{2}-\left[
\beta (d)\right] ^{2}}{8\delta (k)[4\delta (k)-\beta (d)]}\quad (z\in 
\mathbb{U}).
\end{equation}
\end{theorem}

\begin{proof}
By using the well known triangle inequality%
\begin{equation*}
|z_{1}-z_{2}|\geq \left\vert |z_{1}|-|z_{2}|\right\vert
\end{equation*}%
and the inequalities%
\begin{equation*}
(k)_{n}\ n!\geq 2k^{n}\ \ \ ,\forall n\geq 2
\end{equation*}%
\begin{equation*}
(s)_{n}\geq s^{n}
\end{equation*}%
\begin{equation*}
(s)_{n}\leq (s+n-1)^{n}
\end{equation*}%
we obtain for all $z\in \mathbb{U}$%
\begin{eqnarray*}
\left\vert \frac{D_{\lambda }^{n,\gamma }\varphi _{v,b,d}(z)}{z}\right\vert
&=&\left\vert 1+\sum_{m=1}^{\infty }\dfrac{(-d)^{m}}{m!4^{m}(k)_{m}}\left[ 
\frac{\left( 1+\frac{1}{\lambda }\right) _{m}(2)_{m}}{\left( \frac{1}{%
\lambda }\right) _{m}(2-\gamma )_{m}}\right] ^{n}z^{m}\right\vert \\
&\geq &1-\sum_{m=1}^{\infty }\dfrac{\left\vert d\right\vert ^{m}}{%
m!4^{m}(k)_{m}}\left[ \frac{\left( 1+\frac{1}{\lambda }\right) _{m}(2)_{m}}{%
\left( \frac{1}{\lambda }\right) _{m}(2-\gamma )_{m}}\right] ^{n} \\
&\geq &1-\left( \frac{2(1+\lambda )}{(2-\gamma )}\right) ^{n}\dfrac{|d|}{4k}%
\sum_{m=1}^{\infty }\dfrac{|d|^{m-1}}{4^{m-1}(k+1)_{m}(m!)}\left[ \frac{%
\left( 2+\frac{1}{\lambda }\right) _{m-1}(3)_{m-1}}{\left( 1+\frac{1}{%
\lambda }\right) _{m-1}(3-\gamma )_{m-1}}\right] ^{n} \\
&\geq &1-\left( \frac{2(1+\lambda )}{(2-\gamma )}\right) ^{n}\dfrac{|d|}{4k}%
\ \left[ 1+\frac{1}{2}\sum_{m=2}^{\infty }\left[ \frac{|d|}{4(k+1)}\frac{%
\left( \frac{1}{\lambda }+m\right) ^{n}(m+1)^{n}}{\left( 1+\frac{1}{\lambda }%
\right) ^{n}(3-\gamma )^{n}}\right] ^{m-1}\right] \\
&=&\frac{4\mathcal{M}\ k\ \delta (k+1)-\alpha (2k+1)\beta (d)+\frac{\mathcal{%
N}}{8}\left[ \beta (d)\right] ^{2}}{\mathcal{M}\ k[4\delta (k+1)-\beta (d)]}
\end{eqnarray*}%
Which is positive if%
\begin{equation*}
32\mathcal{M}\ k\ \delta (k+1)-8\alpha (2k+1)\beta (d)+\mathcal{N}\left[
\beta (d)\right] ^{2}>0
\end{equation*}%
Similarly,by using the triangle inequality%
\begin{equation*}
|z_{1}+z_{2}|\leq |z_{1}|+|z_{2}|
\end{equation*}%
and the inequalities:%
\begin{equation*}
(s)_{n}\ n!\geq 2s^{n}\ \ ,\forall n\geq 2
\end{equation*}%
\begin{equation*}
(s)_{n}\geq s^{n}
\end{equation*}%
\begin{equation*}
(s)_{n}\leq (s+n-1)^{n}
\end{equation*}%
we obtain for all $z\in \mathbb{U}$%
\begin{eqnarray*}
\left\vert \frac{D_{\lambda }^{n,\gamma }\varphi _{v,b,d}(z)}{z}\right\vert
&=&\left\vert 1+\sum_{m=1}^{\infty }\dfrac{(-d)^{m}}{m!4^{m}(k)_{m}}\left[ 
\frac{\left( 1+\frac{1}{\lambda }\right) _{m}(2)_{m}}{\left( \frac{1}{%
\lambda }\right) _{m}(2-\gamma )_{m}}\right] ^{n}z^{m}\right\vert \\
&\leq &1+\sum_{m=1}^{\infty }\dfrac{\left\vert d\right\vert ^{m}}{%
m!4^{m}(k)_{m}}\left[ \frac{\left( 1+\frac{1}{\lambda }\right) _{m}(2)_{m}}{%
\left( \frac{1}{\lambda }\right) _{m}(2-\gamma )_{m}}\right] ^{n} \\
&\leq &1+\left( \frac{2(1+\lambda )}{(2-\gamma )}\right) ^{n}\dfrac{|d|}{4k}+%
\frac{1}{2}\sum_{m=2}^{\infty }\left[ \frac{|d|}{4(k)}\frac{\left( \frac{1}{%
\lambda }+m\right) ^{n}(m+1)^{n}}{\left( \frac{1}{\lambda }\right)
^{n}(2-\gamma )^{n}}\right] ^{m} \\
&=&\frac{32\left[ \delta (k)\right] ^{2}-\left[ \beta (d)\right] ^{2}}{%
8\delta (k)[4\delta (k)-\beta (d)]}
\end{eqnarray*}%
Thus, the proof is complete.
\end{proof}

\begin{theorem}
\label{theorem2.5}If the parameters $v,b\in \mathbb{R}$ and $d\in \mathbb{C}$
are so constrained that\newline
\begin{equation*}
\displaystyle k>\max \left\{ 0,\frac{\beta (d)}{4(1+\frac{1}{\lambda }%
)^{n}(3-\gamma )^{n}}-1\right\} ,
\end{equation*}%
then the function\newline
\begin{equation*}
\displaystyle D_{\lambda }^{n,\gamma }\varphi _{v,b,d}(z):\mathbb{U}%
\rightarrow \mathbb{C}
\end{equation*}%
satisfies the following inequalities:\newline
\begin{equation}
\displaystyle\left\vert \left( D_{\lambda }^{n,\gamma }\varphi
_{v,b,d}(z)\right) ^{\prime }-\frac{D_{\lambda }^{n,\gamma }\varphi
_{v,b,d}(z)}{z}\right\vert \leqq \frac{\mathcal{N}\ \delta (k+1)\beta (d)}{%
\mathcal{M}\ k[4\delta (k+1)-\beta (d)]}\quad (z\in \mathbb{U}),
\label{eq:2.19}
\end{equation}%
\newline
\end{theorem}

\begin{equation}
\displaystyle\left\vert \frac{z\left( D_{\lambda }^{n,\gamma }\varphi
_{v,b,d}(z)\right) ^{\prime }}{D_{\lambda }^{n,\gamma }\varphi _{v,b,d}(z)}%
-1\right\vert \leqq \frac{8\mathcal{N}\ \delta (k+1)\beta (d)}{32\mathcal{M}%
\ k\ \delta (k+1)-8\alpha (2k+1)\beta (d)+\mathcal{N}\ \left[ \beta (d)%
\right] ^{2}}\quad (z\in \mathbb{U}),  \label{eq:2.20}
\end{equation}%
\newline

\begin{eqnarray}
\displaystyle\frac{4\mathcal{M}\ k\ \delta (k+1)-\alpha (3k+2)\beta (d)}{%
\mathcal{M}\ k[4\delta (k+1)-\beta (d)]} &\leqq &|z\left( D_{\lambda
}^{n,\gamma }\varphi _{v,b,d}(z)\right) ^{\prime }|  \label{eq:2.21} \\
&\leqq &\frac{4\mathcal{M}\ k\ \delta (k+1)+\alpha (k+2)\beta (d)}{\mathcal{M%
}\ k[4\delta (k+1)-\beta (d)]}\quad (z\in \mathbb{U}),
\end{eqnarray}%
\newline

\begin{equation}
\displaystyle\left\vert z^{2}\left( D_{\lambda }^{n,\gamma }\varphi
_{v,b,d}(z)\right) ^{\prime \prime }\right\vert \leqq \frac{\mathcal{N}\
\beta (d)}{2\mathcal{M}\ k}\frac{4\delta (k+1)+\beta (d)}{4\delta
(k+1)-\beta (d)}\quad (z\in \mathbb{U})  \label{eq:2.22}
\end{equation}

\begin{proof}
We first prove the assertation $\left( \ref{eq:2.19}\right) $ of Theorem $%
\left( \ref{theorem2.5}\right) $ when $\lambda >0.$\newline
Indeed, by using the following:\newline
\begin{equation*}
(n-1)!=(1)_{n-1}
\end{equation*}%
\begin{equation*}
(s)_{n}\geq s^{n}
\end{equation*}%
\begin{equation*}
(s)_{n}\leq (s+n-1)^{n}
\end{equation*}%
then we have,%
\begin{eqnarray*}
(n-1)!(k+1)_{n-1} &=&(1)_{n-1}(k+1)_{n-1} \\
&\geqq &(k+1)^{n-1}
\end{eqnarray*}%
we obtain for all $z\in \mathbb{U}$%
\begin{eqnarray*}
\displaystyle\left\vert (D_{\lambda }^{n,\gamma }\varphi
_{v,b,d}^{(z)})^{\prime }-\frac{(D_{\lambda }^{n,\gamma }\varphi
_{v,b,d}^{(z)})}{z}\right\vert &=&\left\vert \sum_{m=1}^{\infty }\frac{%
m(-d)^{m}}{m!4^{m}(k)_{m}}\left[ \frac{\left( 1+\frac{1}{\lambda }\right)
_{m}(2)_{m}}{\left( \frac{1}{\lambda }\right) _{m}(2-\gamma )_{m}}\right]
^{n}z^{m}\right\vert \\
&\leq &\sum_{m=1}^{\infty }\frac{m|d|^{m}}{m!4^{m}(k)_{m}}\left[ \frac{%
\left( 1+\frac{1}{\lambda }\right) _{m}(2)_{m}}{\left( \frac{1}{\lambda }%
\right) _{m}(2-\gamma )_{m}}\right] ^{n} \\
&=&\left( \frac{\left( 1+\frac{1}{\lambda }\right) (2)}{\left( \frac{1}{%
\lambda }\right) (2-\gamma )}\right) ^{n} \\
&&\times \frac{|d|}{4k}\sum_{m=1}^{\infty }\frac{|d|^{m-1}}{%
4^{m-1}(m-1)!(k+1)_{m-1}}\left[ \frac{\left( 2+\frac{1}{\lambda }\right)
_{m-1}(3)_{m-1}}{\left( 1+\frac{1}{\lambda }\right) _{m-1}(3-\gamma )_{m-1}}%
\right] ^{n} \\
&\leq &\left( \frac{2(1+\lambda )}{(2-\gamma )}\right) ^{n}\frac{|d|}{4k}%
\sum_{m=1}^{\infty }\frac{|d|^{m-1}}{4^{m-1}(k+1)^{m-1}}\left[ \frac{\left(
2+\frac{1}{\lambda }\right) _{m-1}(3)_{m-1}}{\left( 1+\frac{1}{\lambda }%
\right) _{m-1}(3-\gamma )_{m-1}}\right] ^{n} \\
&\leq &\left( \frac{2(1+\lambda )}{(2-\gamma )}\right) ^{n}\frac{|d|}{4k}%
\sum_{m=1}^{\infty }\left[ \frac{|d|}{4(k+1)}\frac{\left( \frac{1}{\lambda }%
+m\right) ^{n}(m+1)^{n}}{\left( 1+\frac{1}{\lambda }\right) ^{n}(3-\gamma
)^{n}}\right] ^{m-1} \\
&=&\frac{\mathcal{N}~\ \delta (k+1)\beta (d)}{\mathcal{M}\ k[4\delta
(k+1)-\beta (d)]}
\end{eqnarray*}%
When $\lambda =0$ using the same technique we get%
\begin{equation*}
\left\vert (D_{\lambda }^{n,\gamma }\varphi _{v,b,d}^{(z)})^{\prime }-\frac{%
(D_{\lambda }^{n,\gamma }\varphi _{v,b,d}^{(z)})}{z}\right\vert \leq
\left\vert \frac{~2^{n}\ \delta (k+1)\beta (d)}{\mathcal{M}\ k\ [4\delta
(k+1)-\beta (d)\left( \frac{1+\lambda }{1+m}\right) ^{n}]}\right\vert
\end{equation*}%
\newline
Next, by combining theorem $\left( \ref{theorem2.4}\right) $ and the first
assertion, we immediately see that the second assertion of theorem $\left( %
\ref{theorem2.5}\right) $ holds true for all $z\in \mathbb{U}$ if%
\begin{equation*}
32\mathcal{M}\ k\ \delta (k+1)-8\alpha (2k+1)\beta (d)+\mathcal{N}\left[
\beta (d)\right] ^{2}>0
\end{equation*}%
In order to prove the assertion $\left( \ref{eq:2.21}\right) $ of theorem $%
\left( \ref{theorem2.5}\right) $ we make use of the following inequalities%
\newline
\begin{equation*}
(n+1)\leq 2^{n}
\end{equation*}%
\begin{equation*}
n!=(2)_{n-1}
\end{equation*}%
\begin{equation*}
\dfrac{1}{k}(k)_{n}=(k+1)_{n-1}\quad ,n\in N
\end{equation*}%
\begin{eqnarray*}
n!(k+1)_{n-1} &=&(2)_{n-1}(k+1)_{n-1} \\
&\geq &\left[ 2(k+1)\right] ^{n-1}
\end{eqnarray*}%
\begin{equation*}
(s)_{n}\geq s^{n}
\end{equation*}%
\begin{equation*}
(s)_{n}\leq (s+n-1)^{n}
\end{equation*}%
we thus find that%
\begin{eqnarray*}
|z\left( D_{\lambda }^{n,\gamma }\varphi _{v,b,d}(z)\right) ^{\prime }|
&=&\left\vert z+\sum_{m=1}^{\infty }\dfrac{(m+1)(-d)^{m}}{m!4^{m}(k)_{m}}%
\left[ \frac{\left( 1+\frac{1}{\lambda }\right) _{m}(2)_{m}}{\left( \frac{1}{%
\lambda }\right) _{m}(2-\gamma )_{m}}\right] ^{n}z^{m+1}\right\vert \\
&\leq &1+\sum_{m=1}^{\infty }\dfrac{(m+1)|d|^{m}}{m!4^{m}(k)_{m}}\left[ 
\frac{\left( 1+\frac{1}{\lambda }\right) _{m}(2)_{m}}{\left( \frac{1}{%
\lambda }\right) _{m}(2-\gamma )_{m}}\right] ^{n} \\
&=&1+\left( \frac{2(1+\lambda )}{(2-\gamma )}\right) ^{n}\dfrac{|d|}{2k}%
\sum_{m=1}^{\infty }\dfrac{k(m+1)|d|^{m-1}}{2^{m-1}2^{m}(k)_{m}(m!)}\left[ 
\frac{\left( 2+\frac{1}{\lambda }\right) _{m-1}(3)_{m-1}}{\left( 1+\frac{1}{%
\lambda }\right) _{m-1}(3-\gamma )_{m-1}}\right] ^{n} \\
&\leq &1+\dfrac{\mathcal{N}\beta (d)}{2\mathcal{M}k}\sum_{m=1}^{\infty }%
\left[ \frac{|d|}{4(k+1)}\frac{\left( \frac{1}{\lambda }+m\right)
^{n}(m+1)^{n}}{\left( 1+\frac{1}{\lambda }\right) ^{n}(3-\gamma )^{n}}\right]
^{m-1} \\
&=&\frac{4\mathcal{M}\ k\ \delta (k+1)+\alpha (k+2)\beta (d)}{\mathcal{M}\
k[4\delta (k+1)-\beta (d)]}
\end{eqnarray*}%
which is positive if%
\begin{equation*}
4\mathcal{M}\ k\ \delta (k+1)+\alpha (k+2)\beta (d)>0
\end{equation*}%
Similarly, by using the inequalities%
\begin{equation*}
|z_{1}-z_{2}|\geq \left\vert |z_{1}|-|z_{2}|\right\vert
\end{equation*}%
\begin{equation*}
(s)_{n}\geq s^{n}
\end{equation*}%
\begin{equation*}
(s)_{n}\leq (s+n-1)^{n}
\end{equation*}%
\begin{equation*}
2n!\geq (n+1)
\end{equation*}%
we have \newline
\begin{eqnarray*}
|z\left( D_{\lambda }^{n,\gamma }\varphi _{v,b,d}(z)\right) ^{\prime }|
&=&\left\vert z+\sum_{m=1}^{\infty }\dfrac{(m+1)(-d)^{m}}{m!4^{m}(k)_{m}}%
\left[ \frac{\left( 1+\frac{1}{\lambda }\right) _{m}(2)_{m}}{\left( \frac{1}{%
\lambda }\right) _{m}(2-\gamma )_{m}}\right] ^{n}z^{m+1}\right\vert \\
&\geq &1-\sum_{m=1}^{\infty }\dfrac{(m+1)|d|^{m}}{m!4^{m}(k)_{m}}\left[ 
\frac{\left( 1+\frac{1}{\lambda }\right) _{m}(2)_{m}}{\left( \frac{1}{%
\lambda }\right) _{m}(2-\gamma )_{m}}\right] ^{n} \\
&=&1-\left( \frac{2(1+\lambda )}{(2-\gamma )}\right) ^{n}\dfrac{|d|}{2k}%
\sum_{m=1}^{\infty }\dfrac{k(m+1)|d|^{m-1}}{2^{m-1}2^{m}(k_{m})(m!)}\left[ 
\frac{\left( 2+\frac{1}{\lambda }\right) _{m-1}(3)_{m-1}}{\left( 1+\frac{1}{%
\lambda }\right) _{m-1}(3-\gamma )_{m-1}}\right] ^{n} \\
&=&\leq 1-\dfrac{\mathcal{N}\ \beta (d)}{2\mathcal{M}\ k}\sum_{m=1}^{\infty }%
\left[ \frac{|d|}{4(k+1)}\frac{\left( \frac{1}{\lambda }+m\right)
^{n}(m+1)^{n}}{\left( 1+\frac{1}{\lambda }\right) ^{n}(3-\gamma )^{n}}\right]
^{m-1} \\
&=&\frac{4\mathcal{M}\ k\ \delta (k+1)-\alpha (3k+2)\beta (d)}{\mathcal{M}\
k[4\delta (k+1)-\beta (d)]}
\end{eqnarray*}%
which is positive if%
\begin{equation*}
4\mathcal{M}\ k\ \delta (k+1)-\alpha (3k+2)\beta (d)>0
\end{equation*}%
We now prove the assertion $\left( \ref{eq:2.22}\right) $ of Theroem $\left( %
\ref{theorem2.5}\right) $ by using the following\newline
\begin{eqnarray*}
(m-1)! &\geq &2^{m-2} \\
(k+1)_{m-1} &\geq &(k+1)^{m-1} \\
m+1 &\leq &2^{m} \\
4(n-1)! &\geq &(n+1)\ 
\end{eqnarray*}%
Thus we have%
\begin{eqnarray*}
\left\vert z^{2}\left( D_{\lambda }^{n,\gamma }\varphi _{v,b,d}(z)\right)
^{\prime \prime }\right\vert &=&\left\vert \sum_{m=1}^{\infty }\dfrac{%
(m+1)m\ (-d)^{m}}{m!4^{m}(k)_{m}}\left[ \dfrac{\left( 1+\frac{1}{\lambda }%
\right) _{m}(2)_{m}}{\left( \dfrac{1}{\lambda }\right) _{m}(2-\gamma )_{m}}%
\right] ^{n}z^{m+1}\right\vert \\
&\leq &\sum_{m=1}^{\infty }\dfrac{(m+1)|d|^{m}}{(m-1)!4^{m}(k)_{m}}\left[ 
\frac{\left( 1+\frac{1}{\lambda }\right) _{m}(2)_{m}}{\left( \frac{1}{%
\lambda }\right) _{m}(2-\gamma )_{m}}\right] ^{n} \\
&=&\left( \dfrac{2(1+\lambda )}{2-\gamma }\right) ^{n} \\
&&\times \dfrac{|d|}{k}\left[ \dfrac{1}{2}+\sum_{m=2}^{\infty }\frac{(m+1)}{%
4(m-1)!}\dfrac{|d|^{m-1}\left( 2+\frac{1}{\lambda }\right)
_{m-1}^{n}(3)_{m-1}^{n}}{4^{m-1}(k+1)_{m-1}\left( 1+\frac{1}{\lambda }%
\right) _{m-1}^{n}(3-\gamma )_{m-1}^{n}}\right] \\
&\leq &\left( \dfrac{2(1+\lambda )}{2-\gamma }\right) ^{n}\dfrac{|d|}{k}%
\left[ \dfrac{1}{2}+\sum_{m=2}^{\infty }\left( \frac{|d|}{4(k+1)}\frac{%
\left( \frac{1}{\lambda }+m\right) ^{n}(m+1)^{n}}{\left( 1+\frac{1}{\lambda }%
\right) ^{n}(3-\gamma )^{n}}\right) ^{m-1}\right] \\
&=&\frac{\mathcal{N}\ \beta (d)}{2\mathcal{M}\ k}\frac{4\delta (k+1)+\beta
(d)}{4\delta (k+1)-\beta (d)}
\end{eqnarray*}%
Finally, by combining the inequalities $\left( \ref{eq:2.21}\right) $ and $%
\left( \ref{eq:2.22}\right) $\ we deduce that $(\gamma )$ holds true for all 
$z\in \mathbb{U}.$ Thus the proof is completed.
\end{proof}

\begin{remark}
Taking $n=0$ in the Theorem\ref{theorem2.4}, we obtain a similar result to
that in \cite{7}.
\end{remark}

\begin{remark}
Taking $n=0$ in the Theorem\ref{theorem2.5}, we obtain a similar result to
that in \cite{11}.
\end{remark}

\section{Univelence Criteria}

\setcounter{equation}{0}

In our present investigation, we need these two univalence criteria which we
recall here as Lemmas $\left( \ref{lemma3.1}\right) $ and$\left( \ref%
{lemma3.2}\right) $ see \cite{14},\cite{13}) .\newline

\begin{lemma}
\label{lemma3.1}(see \cite{14} ). Let $\eta $ and $c$ be complex numbers
such that\newline
\begin{equation*}
\mathfrak{R}(\eta )>0\,\mbox{ and }\,|c|\leqq 1\,\,(c\neq -1).\newline
\end{equation*}%
If the function $f\in A$ satisfies the following inequality:%
\begin{equation*}
\left\vert c|z|^{2\eta }+\left( 1+|z|^{2\eta }\right) \frac{zf^{\prime
\prime }(z)}{\eta f^{\prime }(z)}\right\vert \leqq 1\quad \,(z\in \mathbb{U}%
),
\end{equation*}%
then the function $F_{\eta }$ defined by \newline
\begin{equation}
F_{\eta }(z)=\left( \eta \int_{0}^{z}t^{\eta -1}f^{\prime }(t)dt\right)
^{1/\eta }  \label{eq:3.0}
\end{equation}%
is in the class $S$ of normalized univalent functions in $\mathbb{U}$.%
\newline
\end{lemma}

\begin{lemma}
\label{lemma3.2}(see \cite{13}). If $f\in A$ satisfies the following
inequality:\newline
\begin{equation*}
\left( \frac{1-|z|^{2\mathfrak{R}(\mu )}}{\mathfrak{R}(\mu )}\right)
\left\vert \frac{zf^{\prime \prime }(z)}{f^{\prime }(z)}\right\vert \leq
1\quad \quad (z\in \mathbb{U};\mathfrak{R}(\mu )>0),
\end{equation*}%
then, for all $\eta \in \mathbb{C}$ such that $\mathfrak{R}(\eta )\geqq 
\mathfrak{R}(\mu ),$ the function $F_{\eta }$ defined by $\left( \ref{eq:3.0}%
\right) $ is in the class $S$ of normalized univalent functions in $\mathbb{U%
}.$\newline
\end{lemma}

Lemma $\left( \ref{lemma3.3}\right) $ below is a consequence of the
above-mentioned Becker's univalence criterion (see \cite{15}) and the
well-known Schwarz lemma.

\begin{lemma}
\label{lemma3.3}(see \cite{15}). Let the parameter $\zeta \in \mathbb{C}$
and $\theta \in \mathbb{R}$ be so constrained that\newline
\begin{equation*}
\mathfrak{R}(\zeta )\geqq 1,\,\,\theta >1\,\,\mbox{ and }\,\,2\theta |\zeta
|\leqq 3\sqrt{3}.
\end{equation*}%
If the function $q\in \mathcal{A}$ satisfies the following inequality:%
\newline
\begin{equation*}
|zq^{\prime }(z)|\leqq \theta \quad \quad (z\in \mathbb{U}),
\end{equation*}%
then the function $\mathcal{G}_{\zeta }:\mathbb{U}\rightarrow \mathbb{C},$
defined by \newline
\begin{equation*}
\mathcal{G}_{\zeta }(z)=\left[ \zeta \int_{0}^{z}t^{\zeta -1}\left(
e^{q(t)}\right) ^{\zeta }dt\right] ^{1/\zeta },
\end{equation*}%
is in the class $\mathcal{S}$ of normalized univalent functions in $\mathbb{U%
}$.
\end{lemma}

In the past two decades, many authors have determined various sufficient
conditions for the univalence of various general families of integral
operators such as (see\cite{18},\cite{19},\cite{6},\cite{20},\cite{11a},\cite%
{11} and \cite{15}).\newline

In this paper we will focus on some integral operators of the following
types\ $\left( \ref{eq6.1}\right) $,$\left( \ref{eq:1.8}\right) $ and $%
\left( \ref{eq7.1}\right) $involving the normalized forms of the generalized
Bessel functions of the first kind as follows%
\begin{eqnarray*}
\mathcal{H}_{v_{1},...,v_{m},b,d,\mu _{1},...,\mu _{m},\eta }(z) &=&\left[
\eta \int_{0}^{z}t^{\eta -1}\ \dprod\limits_{j=1}^{m}\left( \frac{D_{\lambda
}^{n,\gamma }\varphi _{v_{j,b,d}}(t)}{t}\right) ^{1/\mu _{j}}\ dt\right]
^{1/\eta } \\
\mathcal{F}_{_{v_{1},...,v_{m},b,d,m,\mu }}(z) &=&\left[ (m\ \mu +1)\
\int_{0}^{z}\ \dprod\limits_{j=1}^{m}\ \left( D_{\lambda }^{n,\gamma
}\varphi _{v_{j,b,d}}(t)\right) ^{\mu }dt\right] ^{1/(m\ \mu +1)} \\
\mathcal{G}_{v,b,d,\zeta }\left( z\right) &=&\left[ \zeta
\int_{0}^{z}t^{\zeta -1}\left( e^{D_{\lambda }^{n,\gamma }\varphi
_{v,b,d}(t)}\right) ^{\zeta }dt\right] ^{1/\zeta }
\end{eqnarray*}

\begin{theorem}
\label{theorem3.1} Let the parameters $v_{1},\ldots ,v_{m},b\in \mathbb{R}$
and $d\in \mathbb{C}$ be so constrained that\newline
\begin{equation*}
k_{j}=v_{j}+\dfrac{b+1}{2}>\frac{\beta (d)}{4(1+\frac{1}{\lambda }%
)^{n}(3-\gamma )^{n}}-1\quad \quad (j=1,\ldots ,m).
\end{equation*}%
Consider the functions $D_{\lambda }^{n,\gamma }\varphi _{v_{j,b,d}}:\mathbb{%
U}\rightarrow \mathbb{C}$ defined by 
\begin{equation}
D_{\lambda }^{n,\gamma }\varphi _{v_{j,b,d}}(z)=2^{v_{j}}\Gamma \left( v_{j}+%
\frac{b+1}{2}\right) D_{\lambda }^{n,\gamma }\left[
z^{1-v_{j}/2}w_{v_{j},b,d}(\sqrt{z})\right]  \label{eq:3}
\end{equation}%
Also let%
\begin{equation*}
k=\min \{k_{1},\ldots ,k_{m}\},\mathfrak{R}(\eta )>0,\,\,c\in \mathbb{C}%
\backslash \{-1\}\,\,\mbox{and}\,\,\mu _{j}\in \mathbb{C}\backslash
\{0\}\,\,(j=1,\ldots ,m).
\end{equation*}%
Moreover, suppose that these numbers satisfy the following inequality:%
\newline
\begin{equation*}
|c|+\dfrac{8\mathcal{N}\ \delta (k+1)\beta (d)}{32\mathcal{M}\ k\delta
(k+1)-8\ \alpha (2k+1)\beta (d)+\mathcal{N}\ \left[ \beta (d)\right] ^{2}}%
\sum_{j=1}^{m}\dfrac{1}{|\eta \mu _{j}|}\leqq 1.
\end{equation*}%
Then the function $\mathcal{H}_{v_{1},...,v_{m},b,d,\mu _{1},...,\mu
_{m},\eta }(z):\mathbb{U}\rightarrow \mathbb{C},$ defined by%
\begin{equation}
\mathcal{H}_{v_{1},...,v_{m},b,d,\mu _{1},...,\mu _{m},\eta }(z)=\left[ \eta
\int_{0}^{z}t^{\eta -1}\ \dprod\limits_{j=1}^{m}\left( \frac{D_{\lambda
}^{n,\gamma }\varphi _{v_{j,b,d}}(t)}{t}\right) ^{1/\mu _{j}}\ dt\right]
^{1/\eta }
\end{equation}%
is in the class $\mathcal{S}$ of normalized univalent functions in $\mathbb{U%
}.$
\end{theorem}

\begin{proof}
We begin by setting $\eta =1$ in the function $\mathcal{H}%
_{v_{1},...,v_{m},b,d,\mu _{1},...,\mu _{m},\eta }(z):\mathbb{U}\rightarrow 
\mathbb{C},$ as follows:%
\begin{equation*}
\mathcal{H}_{v_{1},...,v_{m},b,d,\mu _{1},...,\mu
_{m},1}(z)=\int_{0}^{z}\prod_{j=1}^{m}\left( \frac{D_{\lambda }^{n,\gamma
}\varphi _{v_{j,b,d}}(t)}{t}\right) ^{1/\mu _{j}}dt.
\end{equation*}%
First of all, we observe that, since $D_{\lambda }^{n,\gamma }\varphi
_{v_{j,b,d}}\in \mathcal{A},$ that is%
\begin{equation*}
D_{\lambda }^{n,\gamma }\varphi _{v_{j,b,d}}(0)=(D_{\lambda }^{n,\gamma
}\varphi _{v_{j,b,d}})^{\prime }(0)-1=0,
\end{equation*}%
we have%
\begin{equation*}
\mathcal{H}_{v_{1},...,v_{m},b,d,\mu _{1},...,\mu _{m},1}\in \mathcal{A},
\end{equation*}%
that is%
\begin{equation*}
\mathcal{H}_{_{v_{1},...,v_{m},b,d,\mu _{1},...,\mu _{m},1}}(0)=\mathcal{H}%
_{_{v_{1},...,v_{n},b,d,\mu _{1},...,\mu _{m},1}}^{\prime }(0)-1=0.
\end{equation*}%
On the other hand, it is easy to see that%
\begin{equation}
\mathcal{H}_{_{v_{1},...,v_{m},b,d,\mu _{1},...,\mu _{m},1}}^{\prime
}(z)=\prod_{j=1}^{m}\left( \frac{D_{\lambda }^{n,\gamma }\varphi
_{v_{j,b,d}}(z)}{z}\right) ^{1/\mu _{j}}
\end{equation}%
we thus find\newline
\begin{equation*}
\dfrac{z\ \mathcal{H}_{_{v_{1},...,v_{m},b,d,\mu _{1},...,\mu
_{m},1}}^{\prime \prime }(z)}{\mathcal{H}_{_{v_{1},...,v_{m},b,d,\mu
_{1},...,\mu _{m},1}}^{\prime }(z)}=\sum_{j=1}^{m}\dfrac{1}{\mu _{j}}\left( 
\frac{z\left( D_{\lambda }^{n,\gamma }\varphi _{v_{j,b,d}}(z)\right)
^{\prime }}{D_{\lambda }^{n,\gamma }\varphi _{v_{j,b,d}}(z)}-1\right)
\end{equation*}%
Thus, by using the inequality $\left( \ref{eq:2.20}\right) $ of Theorem $%
\left( \ref{theorem2.5}\right) $ 7 for each $v_{j}=(j=1,\ldots ,m),$ we
obtain\newline
\begin{eqnarray*}
\left\vert \dfrac{z\ \mathcal{H}_{k_{v_{1},...,v_{m},b,d,\mu _{1},...,\mu
_{m},1}}^{\prime \prime }(z)}{\mathcal{H}_{_{v_{1},...,v_{m},b,d,\mu
_{1},...,\mu _{m},1}}^{\prime }(z)}\right\vert &\leqq &\sum_{j=1}^{m}\dfrac{1%
}{|\mu _{j}|}\left\vert \frac{z\left( D_{\lambda }^{n,\gamma }\varphi
_{v_{j,b,d}}(z)\right) ^{\prime }}{D_{\lambda }^{n,\gamma }\varphi
_{v_{j,b,d}}(z)}-1\right\vert \\
&\leqq &\sum_{j=1}^{m}\dfrac{1}{|\mu _{j}|}\left( \dfrac{8\mathcal{N}\
\delta (k_{j}+1)\beta (d)}{32\mathcal{M}\ k_{j}\ \delta (k_{j}+1)-8\ \alpha
(2k_{j}+1)\beta (d)+\mathcal{N}\ \left[ \beta (d)\right] ^{2}}\right) \\
&\leqq &\sum_{j=1}^{m}\dfrac{1}{|\mu _{j}|}\left( \dfrac{8\mathcal{N}\
\delta (k+1)\beta (d)}{32\mathcal{M}\ k\delta (k+1)-8\ \alpha (2k+1)\beta
(d)+\mathcal{N}\ \left[ \beta (d)\right] ^{2}}\right)
\end{eqnarray*}%
$\left( z\in \mathbb{U};k=\min \{k_{1},\ldots k_{m}\};k_{j}=v_{j}+\dfrac{b+1%
}{2}>\frac{\beta (d)}{4(1+\frac{1}{\lambda })^{n}(3-\gamma )^{n}}-1\quad
(j=1,\ldots ,m)\right) .$\newline
Here we have used the fact that the function $\phi :\left( \frac{\beta (d)}{%
4(1+\frac{1}{\lambda })^{n}(3-\gamma )^{n}}-1,\infty \right) \rightarrow 
\mathbb{R},$ defined by%
\begin{equation*}
\phi (x)=\dfrac{8\mathcal{N}\ \delta (x+1)\beta (d)}{32\mathcal{M}\ x\
\delta (x+1)-8\ \alpha (2x+1)\beta (d)+\mathcal{N}\ \left[ \beta (d)\right]
^{2}},
\end{equation*}%
is decreasing and , consequently, we have%
\begin{eqnarray*}
&&\dfrac{8\mathcal{N}\ \delta (k_{j}+1)\beta (d)}{32\mathcal{M}\ k_{j}\
\delta (k_{j}+1)-8\ \alpha (2k_{j}+1)\beta (d)+\mathcal{N}\ \left[ \beta (d)%
\right] ^{2}} \\
&\leqq &\dfrac{8\mathcal{N}\ \delta (k+1)\beta (d)}{32\mathcal{M}\ k\delta
(k+1)-8\ \alpha (2k+1)\beta (d)+\mathcal{N}\ \left[ \beta (d)\right] ^{2}}%
,\quad (j=1,\ldots ,m).
\end{eqnarray*}%
Finally, by using the triangle inequality and the assertion of Theorem $%
\left( \ref{theorem3.1}\right) $, we obtain\newline
\begin{eqnarray*}
&&\left\vert c|z|^{2\eta }+\left( 1-|z|^{2\eta }\right) \dfrac{z\mathcal{H}%
_{_{v_{1},...,v_{m},b,d,\mu _{1},...,\mu _{m},1}}^{\prime \prime }(z)}{\eta 
\mathcal{H}_{_{v_{1},...,v_{m},b,d,\mu _{1},...,\mu _{m},1}}^{\prime }(z)}%
\right\vert \\
&\leqq &|c|+\dfrac{8\mathcal{N}\ \delta (k+1)\beta (d)}{32\mathcal{M}\
k\delta (k+1)-8\ \alpha (2k+1)\beta (d)+\mathcal{N}\ \left[ \beta (d)\right]
^{2}}\sum_{j=1}^{m}\dfrac{1}{|\eta \mu _{j}|}\leqq 1,
\end{eqnarray*}%
which, in view of Lemma $\left( \ref{lemma3.1}\right) $, implies that $%
\mathcal{H}_{_{v_{1},...,v_{m},b,d,\mu _{1},...,\mu _{m},\eta }}\in \mathcal{%
S}$.\newline
This evidently completes the proof of Theorem $\left( \ref{theorem3.1}%
\right) $.
\end{proof}

Upon setting%
\begin{equation*}
\mu _{1}=\ldots =\mu _{m}=\mu
\end{equation*}%
in Theorem $\left( \ref{theorem3.1}\right) $, we immediately arrive at the
following application of Theorem $\left( \ref{theorem3.1}\right) $.

\begin{corollary}
\label{cor3.1} Let the parameters $v_{1},\ldots v_{m},b,c,d,\eta $ and $%
k_{j}(j=1,\ldots ,m)$ be prescribed as in Theorem $\left( \ref{theorem3.1}%
\right) $. Also let 
\begin{equation*}
k=\min \{k_{1},\ldots ,k_{m}\}\quad \mbox{ and }\quad \mu \in \mathbb{C}%
\backslash \{0\}.
\end{equation*}%
Moreover, suppose that the functions $D_{\lambda }^{n,\gamma }\varphi
_{v_{j,b,d}}\in \mathcal{A}$ are defined by $\left( \ref{eq:3}\right) $ and
the following inequality:%
\begin{equation*}
|c|+\dfrac{m}{\left\vert \mu \eta \right\vert }\left( \dfrac{8\mathcal{N}\
\delta (k+1)\beta (d)}{32\mathcal{M}\ k\delta (k+1)-8\ \alpha (2k+1)\beta
(d)+\mathcal{N}\ \left[ \beta (d)\right] ^{2}}\right) \leqq 1
\end{equation*}%
holds true. Then the function $\mathcal{H}_{_{v_{1},...,v_{n},b,d,\mu
_{1},...,\mu _{m},\eta }}(z):\mathbb{U}\rightarrow \mathbb{C},$ defined by%
\begin{equation}
\mathcal{H}_{v_{1},...,v_{n},b,d,\mu ,\eta }(z)=\left[ \eta
\int_{0}^{z}t^{\eta -1}\ \dprod\limits_{j=1}^{m}\left( \frac{D_{\lambda
}^{n,\gamma }\varphi _{v_{j,b,d}}(t)}{t}\right) ^{1/\mu }\ dt\right]
^{1/\eta }
\end{equation}%
is in the class$\mathcal{S}$ of normalized univalent functions in $\mathbb{U}
$.
\end{corollary}

Our second result in this section provides sufficient conditions for an
integral operator of the type $\left( \ref{eq:1.8}\right) $. The key tools
in the proof are Lemma $\left( \ref{lemma3.2}\right) $ and the inequality $%
\left( \ref{eq:2.20}\right) $ of Theorem $\left( \ref{theorem2.5}\right) $.

\begin{theorem}
\label{theorem3.2} Let the parameters $v_{1},\ldots ,v_{m},b\in \mathbb{R}$
and $d\in \mathbb{C}$ be so constrained that\newline
\begin{equation*}
k_{j}:=v_{j}+\dfrac{b+1}{2}>\frac{\beta (d)}{4(1+\frac{1}{\lambda }%
)^{n}(3-\gamma )^{n}}-1\quad \quad (j=1,\ldots ,m).
\end{equation*}%
Consider the functions $D_{\lambda }^{n,\gamma }\varphi _{v_{j,b,d}}:\mathbb{%
U}\rightarrow \mathbb{C}$ defined by $\left( \ref{eq:3}\right) $. Also let 
\newline
\begin{equation*}
k=\min \{k_{1},\ldots ,k_{m}\},\quad \mbox{and }\quad \mathfrak{R}(\mu )>0,
\end{equation*}%
Moreover, suppose that these numbers satisfy the following inequality:%
\newline
\begin{equation*}
|\mu |\leqq \dfrac{1}{m}\left( \dfrac{32\mathcal{M}\ k\ \delta (k+1)-8\alpha
(2k+1)\beta (d)+\mathcal{N}\ \left[ \beta (d)\right] ^{2}}{8\mathcal{N}\
\delta (k+1)\beta (d)}\right) \mathfrak{R}(\mu ).
\end{equation*}%
Then the function $\mathcal{F}_{_{v_{1},...,v_{m},b,d,m,\mu }}(z):\mathbb{U}%
\rightarrow \mathbb{C},$ defined by 
\begin{equation}
\mathcal{F}_{_{v_{1},...,v_{m},b,d,m,\mu }}(z)=\left[ (m\ \mu +1)\
\int_{0}^{z}\ \dprod\limits_{j=1}^{m}\ \left( D_{\lambda }^{n,\gamma
}\varphi _{v_{j,b,d}}(t)\right) ^{\mu }dt\right] ^{1/(m\ \mu +1)}
\end{equation}%
is in the class $\mathcal{S}$ of normalized univalent functions in $\mathbb{U%
}.$
\end{theorem}

\begin{proof}
Let us consider the function $\widetilde{\mathcal{F}}%
_{_{v_{1},...,v_{m},b,d,m,\mu }}(z):\mathbb{U}\rightarrow \mathbb{C},$
defined by\newline
\begin{equation*}
\widetilde{\mathcal{F}}_{_{v_{1},...,v_{m},b,d,m,\mu
}}(z)=\int_{0}^{z}\prod_{j=1}^{m}\left( \frac{D_{\lambda }^{n,\gamma
}\varphi _{v_{j},b,d}(t)}{t}\right) ^{\mu }dt.
\end{equation*}%
Observe that,$\widetilde{\mathcal{F}}_{_{v_{1},...,v_{m},b,d,m,\mu }}\in 
\mathcal{A},$ that is ,that%
\begin{equation*}
\widetilde{\mathcal{F}}_{_{v_{1},...,v_{m},b,d,m,\mu }}(0)=\widetilde{%
\mathcal{F}}_{_{v_{1},...,v_{m},b,d,m,\mu }}^{\prime }(0)-1=0.
\end{equation*}%
On the other hand, by using the inequality $\left( \ref{eq:2.20}\right) $ of
Theorem $\left( \ref{theorem2.5}\right) $, the assertion of Theorem $\left( %
\ref{theorem3.2}\right) \ $and the fact that\newline
\begin{gather*}
\frac{8\mathcal{N}\ \delta (k_{j}+1)\beta (d)}{32\mathcal{M}\ k_{j}\ \delta
(k_{j}+1)-8\alpha (2k_{j}+1)\beta (d)+\mathcal{N}\ \left[ \beta (d)\right]
^{2}} \\
\leqq \frac{8\mathcal{N}\ \delta (k+1)\beta (d)}{32\mathcal{M}\ k\ \delta
(k+1)-8\alpha (2k+1)\beta (d)+\mathcal{N}\ \left[ \beta (d)\right] ^{2}}%
\quad (j=1,\ldots ,m),
\end{gather*}%
we have%
\begin{eqnarray*}
&&\dfrac{1-|z|^{2\mathfrak{R}(\mu )}}{\mathfrak{R}(\mu )}\left\vert \dfrac{z%
\widetilde{\mathcal{F}}_{_{v_{1},...,v_{m},b,d,m,\mu }}^{\prime \prime }(z)}{%
\widetilde{\mathcal{F}}_{_{v_{1},...,v_{m},b,d,m,\mu }}^{\prime }(z)}%
\right\vert \\
\newline
&\leqq &\dfrac{|\mu |}{\mathfrak{R}(\mu )}\sum_{j=1}^{m}\left\vert \frac{%
z\left( D_{\lambda }^{n,\gamma }\varphi _{v_{j},b,d}(z)\right) ^{\prime }}{%
D_{\lambda }^{n,\gamma }\varphi _{v_{j},b,d}(z)}-1\right\vert \newline
\\
&\leqq &\dfrac{m|\mu |}{\mathfrak{R}(\mu )}\left( \frac{8\mathcal{N}\ \delta
(k+1)\beta (d)}{32\mathcal{M}\ k\ \delta (k+1)-8\alpha (2k+1)\beta (d)+%
\mathcal{N}\ \left[ \beta (d)\right] ^{2}}\right) \leqq 1\quad \,(z\in 
\mathbb{U}).
\end{eqnarray*}%
Now since\newline
\begin{equation*}
\mathfrak{R}(m\mu +1)>\mathfrak{R}(\mu )\quad \,(m\in \mathbb{N})
\end{equation*}%
and the function $\mathcal{F}_{_{v_{1},...,v_{m},b,d,m,\mu }}(z)$ can be
rewritten in the form:%
\begin{equation*}
\mathcal{F}_{_{v_{1},...,v_{m},b,d,m,\mu }}(z)=\left[ (m\mu
+1)\int_{0}^{z}t^{m\mu }\prod_{j=1}^{m}\left( \dfrac{D_{\lambda }^{n,\gamma
}\varphi _{v_{j},b,d}(t)}{t}\right) ^{\mu }dt\right] ^{1/(m\mu +1)}
\end{equation*}%
which in view of Lemma $\left( \ref{lemma3.2}\right) $, implies that $%
\mathcal{F}_{_{v_{1},...,v_{m},b,d,m,\mu }}(z)\in \mathcal{S}.$ This
evidently completes the proof of Theorem $\left( \ref{theorem3.2}\right) $.
\end{proof}

Choosing $m=1$ in Theorem $\left( \ref{theorem3.2}\right) $, we have the
following result.

\begin{corollary}
\label{cor3.2} Let the paramater $v,b\in \mathbb{R}$ and $d\in \mathbb{C}$
be so constrained that\newline
\begin{equation*}
k=v+\dfrac{b+1}{2}>\frac{\beta (d)}{4(1+\frac{1}{\lambda })^{n}(3-\gamma
)^{n}}-1.
\end{equation*}%
Consider the function $D_{\lambda }^{n,\gamma }\varphi _{v,b,d}:\mathbb{U}%
\rightarrow \mathbb{C}$ defined by $\left( \ref{eq:1}\right) $. Moreover,
suppose that $\mathfrak{R}(\mu )>0$ and\newline
\begin{equation*}
|\mu |\leqq \left( \dfrac{32\mathcal{M}\ k\ \delta (k+1)-8\alpha (2k+1)\beta
(d)+\mathcal{N}\ \left[ \beta (d)\right] ^{2}}{8\mathcal{N}\ \delta
(k+1)\beta (d)}\right) \mathfrak{R}(\mu ).
\end{equation*}%
Then the function $\mathcal{F}_{_{v,b,d,\mu }}(z):\mathbb{U}\rightarrow 
\mathbb{C},$ defined by\newline
\begin{equation}
\mathcal{F}_{_{v,b,d,\mu }}(z)=\left[ (\mu +1)\int_{0}^{z}(D_{\lambda
}^{n,\gamma }\varphi _{v,b,d}(t))^{\mu }dt\right] ^{1/(\mu +1)},
\end{equation}%
is in the class $\mathcal{S}$ of normalized univalent functions in $\mathbb{U%
}.$
\end{corollary}

By applying Lemma $\left( \ref{lemma3.3}\right) $ and the inequality $\left( %
\ref{eq:2.21}\right) $ of Theorem $\left( \ref{theorem2.5}\right) $, we
easily get the following result.\newline

\begin{theorem}
\label{theorem3.3} Let the parameters $k,b\in \mathbb{R}$ and $d,\zeta \in 
\mathbb{C}$ be so constrained that%
\begin{equation*}
k:=v+\dfrac{b+1}{2}>\frac{\beta (d)}{4(1+\frac{1}{\lambda })^{n}(3-\gamma
)^{n}}-1.
\end{equation*}%
Consider the generalized Bessel function $D_{\lambda }^{n,\gamma }\varphi
_{v,b,d}$ defined by $\left( \ref{eq:1}\right) $. If $\mathfrak{R}(\zeta
)\geqq 1$ and%
\begin{equation*}
|\zeta |\leqq \dfrac{3\sqrt{3}\mathcal{M}\ k[4\delta (k+1)-\beta (d)]}{8\ 
\mathcal{M}\ k\ \delta (k+1)+2\ \alpha (k+2)\beta (d)},
\end{equation*}%
then the function $\mathcal{G}_{v,b,d,\zeta }:\mathbb{U}\rightarrow \mathbb{C%
},$ defined by\newline
\begin{equation}
\mathcal{G}_{v,b,d,\zeta }\left( z\right) =\left[ \zeta \int_{0}^{z}t^{\zeta
-1}\left( e^{D_{\lambda }^{n,\gamma }\varphi _{v,b,d}(t)}\right) ^{\zeta }dt%
\right] ^{1/\zeta }
\end{equation}%
is in the class $\mathcal{S}$ of normalized univalent functions in $\mathbb{U%
}$.
\end{theorem}

\begin{remark}
Taking $n=0$ in the above results, we obtain the same results as \cite{11}.
\end{remark}

\section{Special Cases}

\setcounter{equation}{0}

Taking into account the above results, we have the following particular
cases.

\subsection{Bessel Functions}

Choosing $b=d=1$, in $\left( \ref{eq1.1}\right) $ or $\left( \ref{eq2.1}%
\right) $, we obtain the Bessel function $J_{\nu }(z)$ of the first kind of
order $\nu $ defined by $\left( \ref{eq3.1}\right) $. We observe also that%
\begin{eqnarray*}
D_{\lambda }^{n,\gamma }\mathcal{J}_{3/2}(z) &=&D_{\lambda }^{n,\gamma
}\left( \frac{3\sin \sqrt{z}}{\sqrt{z}}-3\cos \sqrt{z}\right) ,\text{ } \\
\text{ }D_{\lambda }^{n,\gamma }\mathcal{J}_{1/2}(z) &=&D_{\lambda
}^{n,\gamma }\left( \sqrt{z}\sin \sqrt{z}\right) \text{ }\text{ }\QTR{md}{and%
}\text{ }\text{ }D_{\lambda }^{n,\gamma }\mathcal{J}_{-1/2}(z)=D_{\lambda
}^{n,\gamma }\left( z\cos \sqrt{z}\right) .
\end{eqnarray*}

\begin{corollary}
\label{cor4.1}Let the function $\mathcal{J}_{\nu }:\mathbb{U}\rightarrow 
\mathbb{C}$ be defined by%
\begin{equation*}
\mathcal{J}_{\nu }(z)=2^{\nu }\Gamma (\nu +1)z^{1-\nu /2}J_{\nu }(\sqrt{z}).
\end{equation*}%
Also let the following assertions hold true:
\end{corollary}

\begin{enumerate}
\item Let $\nu _{1},...,\nu _{m}>-1.25\text{ }\text{ }(m\in \mathbb{N})$.
Consider the functions $D_{\lambda }^{n,\gamma }\mathcal{J}_{\nu _{j}}:%
\mathbb{U}\rightarrow \mathbb{C}$ defined by%
\begin{equation}
D_{\lambda }^{n,\gamma }\mathcal{J}_{\nu _{j}}(z)=2^{\nu _{j}}\Gamma (\nu
_{j}+1)z^{1-\nu _{j}/2}D_{\lambda }^{n,\gamma }J_{\nu _{j}}(\sqrt{z})\text{ }%
\text{ }\text{ }(j=1,...,m).  \label{eq:4.11}
\end{equation}%
Let $\nu =\min \{\nu _{1},...,\nu _{m}\}$ and let the parameters $\eta
,c,\mu _{1},...,\mu _{m}$ be as in Theorem $\left( \ref{theorem3.1}\right) $%
. Moreover, suppose that these numbers satisfy the following inequality: 
\begin{equation*}
|c|+\dfrac{\mathcal{N}\ \delta (v+2)\beta (1)}{4\mathcal{M}\ \left(
v+1\right) \delta (v+2)-\ \alpha (2v+3)\beta (1)+\left( \mathcal{N}\ \left[
\beta (1)\right] ^{2}\right) /8}\sum_{j=1}^{m}\frac{1}{|\eta \mu _{j}|}\leqq
1.
\end{equation*}%
Then the function $\mathcal{H}_{v_{1},...,v_{m},\mu _{1},...,\mu _{m},\eta
}(z):\mathbb{U}\rightarrow \mathbb{C}$, defined by%
\begin{equation}
\mathcal{H}_{v_{1},...,v_{m},\mu _{1},...,\mu _{m},\eta }(z)=\Bigg[\eta
\int_{0}^{z}t^{\eta -1}\prod_{j=1}^{m}\Bigg(\frac{D_{\lambda }^{n,\gamma }%
\mathcal{J}_{\nu _{j}}(t)}{t}\Bigg)^{1/\mu _{j}}dt\Bigg]^{1/\eta }\text{ }%
\text{ },  \label{eq:4.2}
\end{equation}%
is in the class $\mathcal{S}$ of normalized univalent functions in $\mathbb{U%
}$. In the particular case when%
\begin{equation*}
|c|+\frac{28}{233}\frac{1}{|\eta \mu |}\leqq 1,
\end{equation*}%
the function $\mathcal{H}_{3/2,\mu ,\eta }(z):\mathbb{U}\rightarrow \mathbb{C%
},$ defined by%
\begin{equation*}
\mathcal{H}_{3/2,\mu ,\eta }(z)=\Bigg[\eta \int_{0}^{z}t^{\eta -1}\Bigg(%
D_{\lambda }^{n,\gamma }\left( \frac{3\sin \sqrt{t}}{t\sqrt{t}}-\frac{3\cos 
\sqrt{t}}{t}\right) \Bigg)^{1/\mu }dt\Bigg]^{1/\eta }\text{ }\text{ },
\end{equation*}%
is in the class $\mathcal{S}$ of normalized univalent functions in $\mathbb{U%
}$.

\item Let $\nu _{1},...,\nu _{m}>-1.25\text{ }\text{ }(m\in \mathbb{N})$ and
consider the normalized Bessel functions $\mathcal{J}_{\nu _{j}}:\mathbb{U}%
\rightarrow \mathbb{C}$ defined by $\left( \ref{eq:4.11}\right) $. Also let $%
\nu =\min \{\nu _{1},...,\nu _{m}\}$ and $\mathfrak{R}(\mu )>0$ and suppose
that these that numbers satisfy the following inequality:

\begin{equation*}
|\mu |\leqq \frac{1}{m}\left( \dfrac{4\mathcal{M}\ \left( v+1\right) \
\delta (v+2)-\alpha (2v+3)\beta (1)+\left( \mathcal{N}\ \left[ \beta (1)%
\right] ^{2}\right) /8}{\mathcal{N}\ \delta (v+2)\beta (1)}\right) \mathfrak{%
R}(\mu ).
\end{equation*}

Then the function $\mathcal{F}_{_{v_{1},...,v_{m},m,\mu }}(z):\mathbb{U}%
\rightarrow \mathbb{C}$, defined by

\begin{equation}
\mathcal{F}_{_{v_{1},...,v_{m},m,\mu }}(z)=\Bigg[(m\ \mu
+1)\int_{0}^{z}\prod_{j=1}^{m}\big(D_{\lambda }^{n,\gamma }\mathcal{J}_{\nu
_{j}}(t)\big)^{\mu }dt\Bigg]^{1/(m\ \mu +1)},  \label{eq:4.3}
\end{equation}

is in the class $\mathcal{S}$ of normalized univalent functions in $\mathbb{U%
}$. In the particular case when

\begin{equation*}
|\mu |\leqq \frac{89}{20}\mathfrak{R}(\mu ),
\end{equation*}

the function $\mathcal{F}_{1/2,\mu }(z):\mathbb{U}\rightarrow \mathbb{C}$,
defined by

\begin{equation*}
\mathcal{F}_{1/2,\mu }(z)=\Bigg[(\mu +1)\int_{0}^{z}\Big(D_{\lambda
}^{n,\gamma }\left( \sqrt{t}\sin \sqrt{t}\right) \Big)^{\mu }dt\Bigg]%
^{1/(\mu +1)},
\end{equation*}

is in the class $\mathcal{S}$ of normalized univalent functions in $\mathbb{U%
}$.

\item Let $\zeta \in \mathbb{C}$ and $\nu >-1.25$ and consider the
normalized Bessel function $\mathcal{J}_{\nu }(z)$ given by $\left( \ref%
{eq3.1}\right) $. If $\mathfrak{R}(\zeta )\geqq 1$ and%
\begin{equation*}
|\zeta |\leqq \dfrac{3\sqrt{3}\mathcal{M}\ (\nu +1)[4\delta (v+2)-\beta (1)]%
}{8\ \mathcal{M}\ (\nu +1)\ \delta (v+2)+2\ \alpha (v+3)\beta (1)},
\end{equation*}

then the function $\mathcal{G}_{v,\zeta }:\mathbb{U}\rightarrow \mathbb{C}$,
defined by

\begin{equation}
\mathcal{G}_{v,\zeta }(z)=\Bigg[\zeta \int_{0}^{z}t^{\zeta -1}\Big(%
e^{D_{\lambda }^{n,\gamma }\mathcal{J}_{\nu }(t)}\Big)^{\zeta }dt\Bigg]%
^{1/\zeta },  \label{eq:4.4}
\end{equation}

is in the class $\mathcal{S}$ of normalized univalent functions in $\mathbb{U%
}$. In the particular case when $|\zeta |\leqq 1.8959...$, the function $%
\mathcal{G}_{1/2,\zeta }(z):\mathbb{U}\rightarrow \mathbb{C}$, defined by

\begin{equation*}
\mathcal{G}_{1/2,\zeta }(z)=\Bigg[\zeta \int_{0}^{z}t^{\zeta -1}\Big(%
e^{D_{\lambda }^{n,\gamma }\left( \sqrt{t}\sin \sqrt{t}\right) }\Big)^{\zeta
}dt\Bigg]^{1/\zeta },
\end{equation*}

is in the class $\mathcal{S}$ of normalized univalent functions in $\mathbb{U%
}$.
\end{enumerate}

\begin{remark}
Baricz and Frasin proved that the following general integral operators \cite%
{6}:

$\mathcal{H}_{v_{1},...,v_{m},\mu _{1},...,\mu _{m},\eta }(z)$, \text{ }%
\text{ } $\mathcal{F}_{_{v_{1},...,v_{m},m,\mu }}(z)$ \text{ }\text{ } and {%
\ }\text{ } \text{ }$\mathcal{G}_{v,\zeta }(z)$

defined by $\left( \ref{eq:4.2}\right) $, $\left( \ref{eq:4.3}\right) $ and $%
\left( \ref{eq:4.4}\right) $, respectively, are actually univalent for all

$\nu ,\nu _{1},...,\nu _{m}>-0.69098...$

From Corollary $\left( \ref{cor4.1}\right) $,by taking $n=0$ we see that our results $(with%
\text{ }\nu ,\nu _{1},...,\nu _{m}>-1.25)$ are stronger than the
Baricz-Frasin results for the same integral operators (see, for details,\cite%
{6} ).
\end{remark}

\subsection{Modified Bessel Functions}

Taking $b=1$ and $d=-1$ in $\left( \ref{eq1.1}\right) $ or $\left( \ref%
{eq2.1}\right) $, we obtain the modified Bessel function $I_{\nu }(z)$ of
the first kind of order $\nu $ defined by $\left( \ref{eq4.1}\right) $. We
observe also that

\begin{align*}
D_{\lambda }^{n,\gamma }\mathcal{I}_{3/2}(z)& =D_{\lambda }^{n,\gamma
}\left( 3\cos \sqrt{z}-\frac{3\sinh \sqrt{z}}{\sqrt{z}}\right) , \\
D_{\lambda }^{n,\gamma }\mathcal{I}_{1/2}(z)& =D_{\lambda }^{n,\gamma
}\left( \sqrt{z}\sinh \sqrt{z}\right) \text{ }\text{ }\text{ }\text{ }%
\QTR{md}{and}\text{ }\text{ }\text{ }D_{\lambda }^{n,\gamma }\text{ }%
\mathcal{I}_{-1/2}(z)=D_{\lambda }^{n,\gamma }\left( \sqrt{z}\cosh \sqrt{z}%
\right) .
\end{align*}

\begin{corollary}
Let the function $D_{\lambda }^{n,\gamma }\mathcal{I}_{\nu }:\mathbb{U}%
\rightarrow \mathbb{C}$ be defined by%
\begin{equation*}
\mathcal{I}_{\nu }(z)=2^{\nu }\Gamma (\nu +1)z^{1-\nu /2}I_{\nu }(\sqrt{z}).
\end{equation*}%
Also let the following assertions hold true:
\end{corollary}

\begin{enumerate}
\item Let $\nu _{1},...,\nu _{m}>-1.25$ $(m\in \mathbb{N})$. Consider the
functions $\mathcal{I}_{\nu _{j}}:\mathbb{U}\rightarrow \mathbb{C}$ defined
by%
\begin{equation}
D_{\lambda }^{n,\gamma }\mathcal{I}_{\nu _{j}}(z)=2^{\nu _{j}}\Gamma (\nu
_{j}+1)z^{1-\nu _{j}/2}D_{\lambda }^{n,\gamma }I_{\nu _{j}}(\sqrt{z})\text{ }%
\text{ }\text{ }\text{ }(j=1,...,m).  \label{eq:4.5}
\end{equation}%
Let $\nu =\min \{\nu _{1},...,\nu _{m}\}$ and let the parameters $\eta
,c,\mu _{1},...,\mu _{m}$ be as in Theorem 1. Moreover, suppose that these
numbers satisfy the following inequality:%
\begin{equation*}
|c|+\dfrac{\mathcal{N}\ \delta (v+2)\beta (1)}{4\mathcal{M}\ \left(
v+1\right) \delta (v+2)-\ \alpha (2v+3)\beta (1)+\left( \mathcal{N}\ \left[
\beta (1)\right] ^{2}\right) /8}\sum_{j=1}^{m}\frac{1}{|\eta \mu _{j}|}\leqq
1.
\end{equation*}%
Then the function $\mathcal{H}_{v_{1},...,v_{m},\mu _{1},...,\mu _{m},\eta
}(z):\mathbb{U}\rightarrow \mathbb{C}$, defined by%
\begin{equation}
\mathcal{H}_{v_{1},...,v_{m},\mu _{1},...,\mu _{m},\eta }(z)=\Bigg[\eta
\int_{0}^{z}t^{\eta -1}\prod_{j=1}^{m}\Bigg(\frac{D_{\lambda }^{n,\gamma }%
\mathcal{I}_{\nu _{j}}(t)}{t}\Bigg)^{1/\mu _{j}}dt\Bigg]^{1/\eta },
\label{eq:4.6}
\end{equation}%
is in the class $\mathcal{S}$ of normalized univalent functions in $\mathbb{U%
}$. In the particular case when%
\begin{equation*}
|c|+\frac{28}{233}\frac{1}{|\eta \mu |}\leqq 1,
\end{equation*}%
the function $\mathcal{H}_{3/2,\mu ,\eta }(z):\mathbb{U}\rightarrow \mathbb{C%
}$, defined by%
\begin{equation*}
\mathcal{H}_{3/2,\mu ,\eta }(z)=\Bigg[\eta \int_{0}^{z}t^{\eta -1}\Bigg(%
D_{\lambda }^{n,\gamma }\left( \frac{3\cosh \sqrt{t}}{t}-\frac{3\sinh \sqrt{t%
}}{t\sqrt{t}}\right) \Bigg)^{1/\mu }dt\Bigg]^{1/\eta },
\end{equation*}%
is in the class $\mathcal{S}$ of normalized univalent functions in $\mathbb{U%
}$.

\item Let $\nu _{1},...,\nu _{m}>-1.25$ $(m\in \mathbb{N})$ and consider the
normalized modified Bessel functions $\mathcal{I}_{\nu _{j}}:\mathbb{U}%
\rightarrow \mathbb{C}$ defined by $\left( \ref{eq:4.5}\right) $. Let $\nu
=\min \{\nu _{1},...,\nu _{m}\}$ and $\mathfrak{R}(\mu )>0$ and suppose that
these numbers satisfy the following inequality:

\begin{equation*}
|\mu |\leqq \frac{1}{m}\Bigg(\dfrac{4\mathcal{M}\ \left( v+1\right) \ \delta
(v+2)-\alpha (2v+3)\beta (1)+\left( \mathcal{N}\ \left[ \beta (1)\right]
^{2}\right) /8}{\mathcal{N}\ \delta (v+2)\beta (1)}\Bigg)\mathfrak{R}(\mu ).
\end{equation*}

Then the function $\mathcal{F}_{\nu _{1},...,\nu _{m},m,\mu }(z):\mathbb{U}%
\rightarrow \mathbb{C}$, defined by

\begin{equation}
\mathcal{F}_{\nu _{1},...,\nu _{m},m,\mu }(z)=\Bigg[(m\ \mu
+1)\int_{0}^{z}\prod_{j=1}^{m}\big(D_{\lambda }^{n,\gamma }\mathcal{I}_{\nu
_{j}}(t)\big)^{\mu }dt\Bigg]^{1/(m\ \mu +1)},  \label{eq:4.7}
\end{equation}

is in the class $\mathcal{S}$ of normalized univalent functions in $\mathbb{U%
}$ In the particular case when

\begin{equation*}
|\mu |\leqq \frac{89}{20}\mathfrak{R}(\mu ),
\end{equation*}%
the function $\mathcal{F}_{1/2,\mu }(z):\mathbb{U}\rightarrow \mathbb{C}$,
defined by%
\begin{equation*}
\mathcal{F}_{1/2,\mu }(z)=\Bigg[(\mu +1)\int_{0}^{z}\Big(D_{\lambda
}^{n,\gamma }\left( \sqrt{t}\sinh \sqrt{t}\right) \Big)^{\mu }dt\Bigg]%
^{1/(\mu +1)},
\end{equation*}%
is in the class $\mathcal{S}$ of normalized univalent functions in $\mathbb{U%
}$.

\item Let $\zeta \in \mathbb{C}$ and $\nu >-1.25$ and consider the
normalized modified Bessel functions $\mathcal{I}_{\nu }(t)$ given by $%
\left( \ref{eq4.1}\right) $. If $\mathfrak{R}(\zeta )\geqq 1$ and

\begin{equation*}
|\zeta |\leqq \dfrac{3\sqrt{3}\mathcal{M}\ (\nu +1)[4\delta (v+2)-\beta (1)]%
}{8\ \mathcal{M}\ (\nu +1)\ \delta (v+2)+2\ \alpha (v+3)\beta (1)},
\end{equation*}

then the function $\mathcal{G}_{\nu ,\zeta }:\mathbb{U}\rightarrow \mathbb{C}
$, defined by

\begin{equation*}
\mathcal{G}_{\nu ,\zeta }(z)=\Bigg[\zeta \int_{0}^{z}t^{\zeta -1}\Big(%
e^{D_{\lambda }^{n,\gamma }\mathcal{I}_{\nu }(t)}\Big)^{\zeta }dt\Bigg]%
^{1/\zeta },
\end{equation*}

is in the class $\mathcal{S}$ of normalized univalent functions in $\mathbb{U%
}$. In the particular case when $|\zeta |\leqq 1.1809...$, the function $%
\mathcal{G}_{-1/2,\zeta }(z):\mathbb{U}\rightarrow \mathbb{C}$, defined by

\begin{equation*}
\mathcal{G}_{-1/2,\zeta }(z)=\Bigg[\zeta \int_{0}^{z}t^{\zeta -1}\Big(%
e^{D_{\lambda }^{n,\gamma }\left( t\cosh \sqrt{t}\right) }\Big)^{\zeta }dt%
\Bigg]^{1/\zeta },
\end{equation*}

is in the class $\mathcal{S}$ of normalized univalent functions in $\mathbb{U%
}$.
\end{enumerate}

\begin{center}
\textsc{Acknowledgment}
\end{center}

The present investigation was supported under ( The scientific and Research
Project of University of Dammam) Number 2013180.\newline

\end{document}